\documentclass{amsart}

\usepackage{amscd,amssymb,amsmath,graphicx,verbatim}
\newtheorem {thm}{Theorem}[section]
\newtheorem{lem}[thm]{Lemma}
\newtheorem{prop}[thm]{Proposition}
\newtheorem{cor}[thm]{Corollary}
\newtheorem{df}[thm]{Definition}

\newtheorem{ex}[thm]{Example}
\newtheorem{exs}[thm]{Examples}

\begin{document}

\title[Rings in which elements are a sum of a central and a nilpotent element]{Rings in which elements are a sum of a central and a nilpotent element}

\author{Yosum Kurtulmaz}
\address{Yosum Kurtulmaz, Department of Mathematics, Bilkent University,
Ankara, Turkey}\email{ yosum@fen.bilkent.edu.tr}

\author{Abdullah Harmanci}
\address{Abdullah Harmanci, Department of Mathematics, Hacettepe
University, Ankara,~ Turkey}\email{harmanci@hacettepe.edu.tr}

\date{}
\begin{abstract}  In this paper, we introduce a new class of rings whose elements are
a sum of a central element and a nilpotent element, namely, a ring
$R$ is called
 $CN$ if each element $a$ of $R$  has a decomposition $a = c + n$
where $c$ is central and $n$ is nilpotent. In this note, we characterize elements in
$M_n(R)$ and  $U_2(R)$ having CN-decompositions. For any field
$F$, we give examples to show that $M_n(F)$ can not be a CN-ring.
For a division ring $D$, we prove that if  $M_n(D)$ is a CN-ring,
then the cardinality of the center of $D$ is strictly greater than
$n$. Especially, we investigate several kinds of conditions under
which some subrings of full matrix rings over CN rings are CN.

\vskip 0.5cm
\noindent{\bf2010 AMS Mathematics Subject Classification}: 15B33, 15B36, 16S70, 16U99.\\
\noindent{\bf Key words and phrases:} CN ring, CN decomposition, matrix
ring, uniquely nil clean ring, uniquely nil *-clean ring.
\end{abstract}
\maketitle
\section{Introduction}  Throughout this paper all rings are associative with identity unless otherwise stated. Let $R$ be a ring.  Inv$(R)$, $J(R)$, $C(R)$ and nil$(R)$
will denote the group of units, the Jacobson radical, the center
and the set of all nilpotent elements of a ring $R$, respectively.
Recall that in \cite{Di}, uniquely nil clean rings are defined. An
element $a$ in a ring $R$ is called {\it uniquely nil clean} if
there is a unique idempotent $e\in R$ such that $a-e$ is
nilpotent. The ring $R$ is {\it uniquely nil clean} if each of its
elements is uniquely clean. It is proved that in a uniquely nil
clean ring, every idempotent is central. Also a uniquely nil clean
ring $R$ is called {\it uniquely strongly nil clean} \cite{HTY} if
$a$ and $e$ commute. Strongly nil cleanness and uniquely strongly
nil cleanness are equivalent by \cite{Di}. Let $R$ be a (*)-ring.
In \cite{LZ}, $a\in R$ is called {\it uniquely strongly nil
*-clean ring} if there is a unique projection $p\in R$, i.e., $p^2
= p = p^*$, and $n\in$ nil$(R)$ such that $a = p + n$ and $pn =
np$. $R$ is called a {\it uniquely strongly nil *-clean ring} if
each of its elements is uniquely strongly nil *-clean. Another
version of the notion of clean rings is that of CU rings. In
\cite{CU}, an element $a\in R$ is called a {\it CU element} if
there exist $c\in C(R)$ and $n\in$ nil$(R)$ such that $a = c + n$.
The ring $R$ is called {\it CU} if each of its elements is CU.
Motivated by these facts, we investigate basic properties of rings
in which every element is the sum of a central element and a
nilpotent element.

In what follows, $\Bbb{Z}_n$ is the ring of integers modulo $n$
for some positive integer $n$. Let $M_n(R)$ denote the full matrix
ring over $R$ and $U_n(R)$ stand for the subring of $M_n(R)$
consisting of all $n\times n$ upper triangular matrices. And in
the following, we give definitions of  some other subrings of
$U_n(R)$ to discuss in the sequel whether they satisfy CN
property:
$$D_n(R)=\{(a_{ij})\in M_n(R) \mid ~\mbox{all diagonal entries
of}~ (a_{ij})~\mbox{are equal}\},$$
$$ V_n(R) = \left\{\sum^n_{i=j}\sum^n_{j=1}a_je_{(i-j+1)i}\mid a_j\in R\right\},$$
 $$ V^k_n(R) = \left\{
\sum^n_{i=j}\sum^k_{j=1}x_je_{(i-j+1)i} +
\sum^{n-k}_{i=j}\sum^{n-k}_{j=1}a_{ij}e_{j(k+i)} : x_j, a_{ij}\in
R\right\}$$ where $x_i\in R$, $a_{js}\in R$, $1\leq i\leq k$,
$1\leq j\leq n-k$ and $k + 1\leq s\leq n$,
$$ D^k_n(R)=\{\left\{
\sum^k_{i=1}\sum^n_{j=k+1}a_{ij}e_{ij} +
\sum^n_{j=k+2}b_{(k+1)j}e_{(k+1)j} + cI_n\mid a_{ij}, b_{ij}, c\in
R\right\}$$ where $k =[n/2]$, i.e., $k$ satisfies $n = 2k$ when
$n$ is an even integer, and $n = 2k+1$ when $n$ is an odd integer,
and
 $$D^{\Bbb Z}_n(R) = \left\{(a_{ij})\in U_n(R)\mid a_{11} = a_{nn}\in \Bbb Z,
 a_{ij}\in R, \{i, j\}\subseteq \{2, 3,\dots, n-1\}\right\}.$$

 \section{Basic Properties and Examples}

\begin{df}{\rm Let $R$ be a ring with identity. An element $a\in R$ is called {\it CN} or {\it it has a CN-decomposition}
 if $a = c + n$, where $c\in C(R)$ and $n\in$ nil$(R)$. If every element of $R$ has a CN decomposition, then $R$ is called {\it a CN ring}.}
\end{df}
 We present some examples to illustrate the concept of CN property for rings.
\begin{ex}{\rm \begin{enumerate}\item[(1)]   Every commutative ring is CN. \item[(2)]
Every nilpotent element in a ring $R$ has a CN
decomposition.\item[(3)] For a field $F$ and for any positive
integer $n$, $D_n(F)$ is a CN ring.\end{enumerate}}
\end{ex}

\begin{prop} Let $R$ be a ring and $n$ a positive integer. Then $A\in M_n(R)$ has a
CN decomposition if and only if for each $P\in GL_n(R)$,
$PAP^{-1}\in M_n(R)$ has a CN decomposition.
\end{prop}
\begin{proof} Assume that $A\in M_n(R)$ has a CN decomposition $A = C + N$ where $C\in C(M_n(R))$ and $N\in$nil$(M_n(R))$. Then $PAP^{-1} = PCP^{-1} + PNP^{-1}$ is a CN decomposition of $PAP^{-1}$ since $PCP^{-1} = C\in C(M_n(R))$ and it is obvious that $PNP^{-1}\in$ nil$(M_n(R)) $ . Conversely, suppose that $PAP^{-1}$ has a CN decomposition $PAP^{-1} = C + N$. Then $A = P^{-1}CP + P^{-1}NP$ is the CN decomposition of $PAP^{-1}$.
\end{proof}

Let $R$ be a commutative ring and $n$ a positive integer. The
following result gives us a way to find out whether $A\in M_n(R)$
has a CN decomposition. Note that it is easily shown that for a commutative ring $A\in C(M_n(R))$ if and only if $A = cI_n$ for some $c\in R$.
\begin{thm}\label{har} Let $R$ be a commutative ring.
Then $A\in M_n(R)$ has a CN decomposition if and only if $A -
cI_n\in$ nil$(M_n(R))$ for some $c\in R $.
\end{thm}

\begin{proof} Assume that $A\in M_n(R)$ has a CN decomposition.
By assumption there exists $c\in R$ such that $A - cI_n\in$ nil$(M_n(R))$. Conversely, suppose that for any $A\in M_n(R)$, there
exists $c\in R$ such that $A - cI_n\in$ nil$(M_n(R))$. Since $cI_n$ is
central in $ M_n(R)$, $A\in M_n(R)$ has a CN decomposition.
\end{proof}
 {\bf Remark}. Let $R$ be a commutative ring. Then $A\in M_n(R)$ is a nilpotent matrix if and only if all eigenvalues of $A$ are zero. A ring $R$ is {\it reduced} if $R$ has no nonzero nilpotent element.
 Hence we have.
\begin{cor} Let $R$ be a commutative reduced ring and $n$ a positive integer. Then $A\in M_n(R)$ has a
CN decomposition if and only if the only eigenvalue for $A-cI_n$
is $0$ for some $c\in R$.
\end{cor}

\begin{prop}\label{id} Let $R$ be a commutative ring. Then $U_2(R)$ is a CN
ring if and only if for any $a, b\in R$, there exists $c\in R$
such that $a-c$, $b-c\in$ nil$(R)$.
\end{prop}
\begin{proof} Let $A = \begin{bmatrix}a&0\\0&b\end{bmatrix}\in M_2(R)$ has CN decomposition if and only if there exist $C = \begin{bmatrix}c&0\\0&c\end{bmatrix}\in C(M_2(R))$ and $N = \begin{bmatrix}x&y\\0&z\end{bmatrix}\in$ nil$(M_2(\Bbb R))$ such that $A = C + N$. Since $N\in$ nil$(M_2(R))$  if and only if $x$, $z\in$ nil$(R)$, $A = C + N$ is the CN decomposition of $A$ if and only if there exists $c\in R$ such that $A - cI\in$ nil$(M_2(R))$ if and only if $a -c$, $b - c\in$ nil$(R)$.
\end{proof}
\begin{ex}{\rm Let $R=\Bbb Z$ and  $A =
\left[\begin{array}{cc} 3 & 0
\\ 0 & 5 \\\end{array}\right]\in U_2(R)$. Then there is no $c\in \Bbb Z $ such that $3-c$ and $5-c$ are nilpotent. By Proposition \ref{id}, $ U_2(\Bbb Z)$ is not CN.}
\end{ex}
\begin{thm}\label{fie} Let $R$ be a commutative local ring. If $M_2(R)$ is a CN ring, then $R/J(R)$ is not
isomorphic to $\Bbb Z_2$.
\end{thm}
\begin{proof} Assume that $M_2(R)$ is a CN ring.
Suppose that $R/J(R)$ is isomorphic to $\Bbb
Z_2$ and we get a contradiction. Let $A = \begin{bmatrix}1&0\\0&0\end{bmatrix}\in M_2(R)$ and
$f(c) = det(A - c I_2)$ be the characteristic polynomial of $A$. Then $f(c) = c(c - 1)\in$ nil$(R)$. By Proposition \ref{id}, $1 - c$ and $c$  are nilpotent. Since $1 = c + ( 1 - c)$, By hypothesis, $c$ or $1 - c$ is invertible, therefore $c\in J(R)$ or $1 - c\in J(R)$. This is a contradiction.
\end{proof}
In \cite{CU}, Chen and at al. defined and studied {\it CU} rings. Let $R$ be a ring. An element $a\in R$ has a {\it CU}-decomposition if $a = c+u$ for some $c\in C(R)$ and $u\in U(R)$. A ring $R$ is called {\it CU}, if every element of $R$ has a {\it CU}-decomposition.
\begin{prop}\label{first} Every CN ring is CU.
\end{prop}
\begin{proof} Let $R$ be a CN ring and $a\in R$. By assumption $a+1=c+n$ for some $c\in C(R)$ and $n\in N(R)$. Hence $a=c+(n-1)$ is a {\it CU} decomposition of $a$.
\end{proof}
\begin{thm} \label{kil} Let $R$ be a division ring and $n$ a positive integer. If $M_n(R)$ is a CN ring, then $|C(R)| > n$.
\end{thm}
\begin{proof} Assume that $|C(R)| < n$.  Consider $A$ as a diagonal matrix
which has the property that each element of $C(R)$ is one of the
diagonal entries of $A$. For such a matrix $A$ there is no $c\in
C(R)$ for which $A - cI$ is a unit. Hence $M_n(R)$ is not a CU ring. By Proposition \ref{first}, $M_n(R)$ can not be a CN
ring. This contradicts hypothesis. So $|C(R)| > n$.
\end{proof}

The converse of Proposition \ref{first} is not true in general.
\begin{ex} {\rm Let $\Bbb{H}=\{ a+bi+cj+dk|a,b,c,d\in \Bbb{R} \}$ be the ring of real quaternions, where
$i^2=j^2=k^2=ijk=1$ and $ij = -ji, ik = -ki, jk = -kj$.
 $\Bbb{H}$ is a noncommutative  division ring. Note that $C(\Bbb{H})=\Bbb{R}$ and $nil(\Bbb{H})=0 .$
 Let $a\in \Bbb{H}$. If $a=0$, then $0=1+(-1)$ is the  {\it CU}-decomposition. If $a\neq 0$, then $a=0+a$ is the  {\it CU}-decomposition of $a$. Hence $\Bbb{H}$ is a {\it CU} ring.
 On the other hand there is no CN decomposition of $i\in \Bbb{H}$. Hence it is not a CN ring.}
\end{ex}

\begin{ex} {\rm Let $D$ be a division ring and consider the ring $D_2(D)$.
The ring $D_2(D)$ is a noncommutative local ring, and so it is a
CU-ring, but not a CN ring.}
\end{ex}

For a positive integer $n$, one may suspect that if $R$ is a {\it
CN} ring then the matrix ring $M_n(R)$ is also {\it CN}. The
following examples shows that this is not true in general. Also
whether or not $M_n(R)$ to be a CN ring does not depend on the
cardinality of $C(R)$ comparing with $n$, that is, $|C(R)|\geq n$
or $|C(R)|< n$.
\begin{exs}\label{harman} {\rm (1) Since $\Bbb Z$ is commutative, it is
a $CN$ ring. But $R = M_2(\Bbb Z)$ is not a $CN$ ring.\\ (2) $R = M_2(\Bbb Z_3)$ is not a CN ring.\\(3) $R = M_3(\Bbb Z_2)$ is not a CN ring.}
\end{exs}
\begin{proof} (1) Consider $A = \begin{bmatrix}1 & 2
\\ 3 & 6 \end{bmatrix}\in M_2(\Bbb Z)$ which is neither central nor nilpotent.
Let $C = \begin{bmatrix}r&0\\0&r\end{bmatrix}\in C(M_2(\Bbb Z))$ and
 $N = \begin{bmatrix}x&y\\z&t\end{bmatrix}\in$nil$(M_2(\Bbb Z))$ with
 $A = C + N$. Then $x + t = 0$ and $zy = xt$. This is a contradiction.
 Hence $A$ does not have CN decomposition.\\(2) Let $A =
\begin{bmatrix}1 & 0\\ 0 & 0\end{bmatrix}\in M_2(\Bbb Z_3)$ which is
neither central nor nilpotent. Assume that $A$ has CN decomposition with
$A = C + N$ where $C =
\begin{bmatrix}a &0\\ 0& a\end{bmatrix}\in C(M_2(\Bbb Z_3))$ and $N =
\begin{bmatrix}x&y\\t&u\end{bmatrix}\in$ nil$(M_2(\Bbb Z_3))$. $A = C + N$ implies
 $1 = a + x$, $0 = a + u$ and $y = t = 0$. These equalities do not satisfied in $\Bbb Z_3$.
 For if $a = 0$, then $x = 1$; if $a = 1$, then $x = 0$ and $u = 2$; if $a = 2$, then $x = 2$ and $u=1$.
All these lead us a contradition. Hence $M_2(\Bbb Z_3)$ is not a
CN ring.  \\(3) Let $A = \begin{bmatrix}1 & 0&0\\ 0 & 0&0
\\0&0&0\end{bmatrix}\in M_3(\Bbb Z_2)$ which is neither central
nor nilpotent. Assume that $A$ has CN decomposition with $A = C +
N$ where $C =
\begin{bmatrix}a &0&0\\ 0& a&0\\ 0&0&a\end{bmatrix}\in C(M_3(\Bbb Z_2))$ and $N =
\begin{bmatrix}x&y&z \\t&u&v\\k&l&m\end{bmatrix}\in$ nil$(M_3(\Bbb Z_2))$. $A = C + N$
 implies $1 = a + x$, $0 = a + u$, $0 = a + m$ and $y = z = v = t = k = l = 0$.
 These equalities do not satisfied in $\Bbb Z_2$. Hence $ M_3(\Bbb Z_2)$ is not a CN ring.
 In fact, assume that $1 = a + x$ holds in $\Bbb Z_2$. There are two cases for $a$. $a= 0$ or $a = 1$.
 If $a = 1$ then $x = 0$ and $u = 1$. $N$ being nilpotent implies $u = 1$ is nilpotent. A contradiction.
 Otherwise, $a = 0$. Then $x = 1$. Again $N$ being nilpotent implies $x = 1$ is nilpotent. A contradiction.
  Thus $M_3(\Bbb Z_2)$ is not a CN ring.
\end{proof}

In spite of the fact that $U_n(R)$ need not be {\it CN} for any
positive integer $n$, there are {\it CN} subrings of $U_n(R)$.

\begin{prop}\label{bes} For a ring $R$ and an integer $n \geq 1$, the following are equivalent:\\{\rm (1)}
$R$ is CN.
\\{\rm (2)} $D_n(R)$ is CN.\\ {\rm (3)} $D^k_n(R)$ is CN.\\{\rm (4)} $V_n(R)$ is CN.\\{\rm (5)} $V^k_n(R)$ is CN.
\end{prop}
\begin{proof} Note that the elements of $D_n(R)$, $D^k_n(R)$, $V_n(R)$ and $V^k_n(R)$  having zero as diagonal
entries are nilpotent. To complete the proof, it is enough to show  (1) holds if and only if (2) holds for $n = 4$.
The other cases are just a repetition.\\ (1) $\Rightarrow$ (2) Let $A = \begin{bmatrix} a_1&a_2 & a_3 &a_4\\
0&a_1&a_5&a_6\\0&0&a_1&a_7\\0&0&0&a_1\end{bmatrix}\in D_4(R)$. By (1), there exist $c\in C(R)$ and $n\in$ nil$(R)$ such that $a_1 = c + n$.\\ Let $C = \begin{bmatrix} c&0&0&0\\
0&c&0&0\\0&0&c&0\\0&0&0&c\end{bmatrix}$ and $N = \begin{bmatrix} n&a_2 & a_3 &a_4\\
0&n&a_5&a_6\\0&0&n&a_7\\0&0&0&n\end{bmatrix}$. Then $C\in C(V_n(R))$ and $N\in$ nil$(D_n(R))$.\\ (2) $\Rightarrow$ (1) Let $a\in R$. By (2) $A = \begin{bmatrix} a&0 &0 &0\\0&a&0&0\\
0&0&a&0\\0&0&0&a\end{bmatrix}\in D_4(R)$ has a CN decomposition $A = C + N$ where $C = \begin{bmatrix} c&0 &0&0\\0&c&0&0\\0&0&c&0\\0&0&0&c\end{bmatrix}\in C(D_4(R))$ and $N = \begin{bmatrix} n&* &*&*\\0&n&*&*\\0&0&n&*\\0&0&0&n\end{bmatrix}\in C(D_n(R))$. Then $a = c + n$ with $c\in C(R)$ and $n\in$ nil$(R)$.
\end{proof}
\begin{lem} Every homomorphic image of CN ring is CN ring.
\end{lem}
\begin{proof} Let $f : R\rightarrow S$ be an epimorphism of rings with $R$ CN ring.
Let $s = f(x)\in S$ with $x\in R$. There exist $c\in C(R)$ and
$n\in$ nil$(R)$ such that $x = c + n$. Since $f$ is epic, $f(c)\in
C(S)$ and $f(n)\in$ nil$(R)$. Hence $s = f(c) + f(n)$ is CN
decomposition of $s$.
\end{proof}
\begin{prop}\label{direct} Let $R = \prod_{i\in I} R_i$ be a direct product of rings. $R$ is CN if and only if $R_i$ is CN for each $i\in I$.
\end{prop}
\begin{proof} We may assume that $I = \{1, 2\}$ and $R = R_1\times R_2$.
Note that $C(R) = C(R_1)\times C(R_2)$ and $nil(R) = nil(R_1)\times nil(R_2)$.
\\Necessity: Let $r_1\in R_1$. Then $(r_1, 0) = (c_1, c_2) + (n_1, n_2)$ where  $(c_1, c_2)\in C(R)$ and
$(n_1, n_2)\in nil(R)$.
 Hence $r_1 = c_1 + n_1$ is the CN decomposition of $r_1\in R_1$. So $R_1$ is CN. A similar proof takes care for $R_2$ be CN.
  \\
Sufficiency: Assume that $R_1$ and $R_2$ are CN. Let $(r_1,
r_2)\in R$. By assumption $r_1$ and $r_2$ have CN decompositions
$r _1 = c_1 + n_1$ and $r_2 = c_2 + n_2$ where
 $c_1$ is central in $R_1$, $n_1$ is nilpotent in $R_1$ and $c_2$ is central in $R_2$, $n_2$
is nilpotent in $R_2$. Hence $(r_1, r_2)$ has a  CN  decomposition
$(r_1, r_2) = (c_1, c_2) + (n_1, n_2)$. This completes the proof.
\end{proof}

Let $R$ be a ring and $D(\Bbb{Z},R)$ denote the \textit{Dorroh
extension} of $R$ by the ring of integers $\Bbb{Z}$ (see \cite{Dr}). Then
$D(\Bbb{Z},R)$ is the ring defined by the direct sum
$\Bbb{Z}\oplus R$ with componentwise addition and multiplication
$(n,r)(m,s)=(nm,ns+mr+rs)$ where $(n,r),~(m,s)\in D(\Bbb{Z},R)$.
It is clear that $C(D(\Bbb{Z},R))= \Bbb{Z}\oplus C(R)$. The
identity of $D(\Bbb{Z}, R)$ is $(1, 0)$ and the set of all
nilpotent elements is nil$(D(\Bbb Z, R))=\{(0,r)\mid r\in
\mbox{nil}(R)\}$.
\begin{thm}
Let $R$ be a ring.  Then $R$ is a CN ring if and only if
$D(\Bbb{Z},R)$ is CN.
\end{thm}
\begin{proof} Assume that $R$ is CN.
Let $(a,r)\in D(\Bbb{Z},R)$. Since $R$ is a CN ring, $r = c + n$
for some $c\in C(R)$ and $n\in$ nil$(R)$. Then $(a, r) = (a, c) +
(0, n)$ is the CN decomposition of $(a, r)$.
 Conversely, let $r\in R$. Then $(0, r) = (a, c) + (0, s)$ as a CN decomposition where $(n, c)\in C(D(\Bbb{Z},R))$ and $(0, n)\in$ nil$(D(\Bbb{Z},R))$. Then $c\in C(R)$  and $s\in$ nil$(R)$. It follows that $r = c + s$ is the CN decomposition of  $r$. Hence $R$ is CN.
\end{proof}

Let $R$ be a ring and $S$ a subring of $R$ and
$$T[R, S] = \{ (r_1, r_2, \cdots , r_n, s, s, \cdots) : r_i\in R, s\in S, n\geq 1, 1\leq i\leq n\}.$$ Then $T[R, S]$ is
a ring under the componentwise addition and multiplication. Note
that nil$(T[R,S])= T[$nil$(R)$, nil$(S)]$ and $C([T,S]) = T[C(R),
C(R)\cap C(S)]$.
\begin{prop}$R$ be a ring and $S$ a subring of $R$. Then the
following are equivalent.
\begin{enumerate}
\item $T[R, S]$ is CN.
\item $R$ and $S$ are CN.
\end{enumerate}
\end{prop}
\begin{proof}
(1) $\Rightarrow$ (2) Assume that  $T[R, S]$ is a CN ring. Let
$a\in R$ and $X=(a,0,0,\dots)\in T[R,S]$. There exist a central element $C=(r_1, r_2, \cdots , r_n, s, s,
\cdots)$ and a nilpotent element $N = (s_1, s_2, \cdots , s_k, t,
t, \cdots)$ in  $T[R, S]$ such that $X= C + N$. Then $r_1$ is in
the center of $R$ and $s_1$ is nilpotent in $R$ and $a=r_1+s_1$ is
the CN decomposition of $a$. Hence $R$ is CN. Let $s\in S$. By considering $Y = (0, s, s, s, \cdots)\in T[R, S]$,
it can be seen that $s$ has a CN decomposition.\\
(2) $\Rightarrow$ (1) Let $R$ and $S$ be CN rings and
$Y=(a_{1},a_{2},\cdots, a_{m}, s,s,s,\cdots)$ be an arbitrary
element in $T[R, S]$. Then there exist $c_{i}\in C(R)$, $1\leq
i\leq m$, $c\in C(R)\cap C(S)$ and $n_{i}\in$ nil$(R)$, $1\leq
i\leq m$, $t\in$ nil$(S)$ and such that $a_{i}= c_{i}+n_{i}$ for
all $1\leq i\leq m$ and $s = c+t$. Let $C =
(c_{1},c_{2},\cdots,c_{m},c,c,\cdots)$ and $N =
(n_{1},n_{2},\cdots,n_{m},t,t,\cdots)$. It is obvious that $C\in
C(T[R, S])$ and $N\in$ nil$(T[R, S])$. Hence $Y = C + N$ is a CN
decomposition of $Y$.
\end{proof}

\section{Some CN subrings of matrix rings}
In this section, we study some subrings of full matrix rings
whether or not they are CN rings. We first determine nilpotent and
central elements of so-called subrings of matrix rings.

{\bf The rings $L_{(s,t)}(R)$ :} Let $R$ be a ring, and $s,
t\in C(R)$. Let $L_{(s,t)}(R) = \left
\{\begin{bmatrix}a&0&0\\sc&d&te\\0&0&f\end{bmatrix}\in M_3(R)\mid
a, c, d, e, f\in R\right \}$, where the operations are defined as
those in $M_3(R)$. Then $L_{(s,t)}(R)$ is a subring of $M_3(R)$.
\begin{lem}\label{com} Let $R$ be a ring, and let $s, t$ be in the center of $R$. Then the following hold.
\begin{enumerate}\item[(1)] The set of all nilpotent elements of $L_{(s, t)}(R)$ is \begin{center} nil$(L_{(s, t)}(R)) = \left\{\begin{bmatrix}a&0&0\\sc&d&te\\0&0&f\end{bmatrix}\in L_{(s,
t)}(R)\mid a, d, f\in  nil(R), c, e\in R\right
\}$.\end{center}\item[(2)] The set of all central elements of
$L_{(s, t)}(R)$ is $$C(L_{(s, t)}(R))) =
\left\{\begin{bmatrix}a&0&0\\sc&d&te\\0&0&f\end{bmatrix}\in H_{(s,
t)}(R)\mid sa = sd, td = tf, a, d, f\in  C(R)\right \}.$$
\end{enumerate}
\end{lem}
\begin{proof} (1) Let $A = \begin{bmatrix}a&0&0\\sc&d&te\\0&0&f\end{bmatrix}\in$ nil$( L_{(s,
t)}(R)$. Assume that $A^n = 0$. Then $a^n = d^n = f^n = 0$. Conversely, Let $A = \begin{bmatrix}a&0&0\\sc&d&te\\0&0&f\end{bmatrix}\in ( L_{(s,
t)}(R)$ with $a^{n_1} = 0$, $d^{n_1} = 0$ and $f^{n_1}= 0$ and
$n =$ max$\{n_1, n_2, n_3\}$. Then $A^{n+1} = 0$.\\(2) Let
$A = \begin{bmatrix}a&0&0\\sc&d&te\\0&0&f\end{bmatrix}\in C(L_{(s, t)}(R)))$ and
$B = \begin{bmatrix}1&0&0\\s&0&t\\0&0&0\end{bmatrix}\in L_{(s, t)}(R))$. By $AB = BA$ implies $sc + sd = sa$ and $td = tf$..................................(*)\\
Let $C = \begin{bmatrix}0&0&0\\s&0&t\\0&0&1\end{bmatrix}\in L_{(s, t)}(R))$.\\
$AC = CA$ implies $sa = sd$ and $dt + te = tf$........................................(**).
  \\(*) and (**) implies $sa = sd$ and $tf = td$. For the converse inclusion,\\
  let $A = \begin{bmatrix}a&0&0\\0&d&0\\0&0&f\end{bmatrix}\in L_{(s, t)}(R)$ with
  $sa = sd$, $td = tf$ and  $a$, $d$, $f\in  C(R)$. Let
  $B = \begin{bmatrix}x&0&0\\sy&z&tu\\0&0&v\end{bmatrix}\in L_{(s, t)}(R)$.
  Then $AB = \begin{bmatrix}ax&0&0\\sdy&dz&e\\0&0&tdu\end{bmatrix}$,
  $BA = \begin{bmatrix}xa&0&0\\sya&zd&tuf\\0&0&vf\end{bmatrix}$.
  By the conditions;  $sa = sd$, $td = tf$, $sc=0$, $te=0$ and  $a$, $d$, $f\in  C(R)$,
   $AB = BA$ for all $B\in L_{(s, t)}(R)$. Hence $A\in C(L_{(s, t)}(R))$.
\end{proof}
Consider following subrings of $L_{(s, t)}(R)$. $$V_2(L_{(s,t)}(R)) =
\left\{\begin{bmatrix}a&0&0\\0&a&te\\0&0&a\end{bmatrix}\in L_{(s,
t)}(R)\mid a,e\in R\right\}$$

$$C(L_{(s,t)}(R)) =
\left\{\begin{bmatrix}a&0&0\\sc&d&te\\0&0&f\end{bmatrix}\in L_{(s,
t)}(R)\mid a, d, f\in C(R), c, e\in R, sa = sd, td = tf\right\}$$ It is easy to check that $V_2(L_{(s,t)}(R))$ and $C(L_{(s,t)}(R))$ are subrings of $L_{(s,
t)}(R)$.

\begin{prop} Let $R$ be a ring. Following hold:\begin{enumerate}\item[(1)] $R$ is a CN ring if and only if $V_2(L_{(s,t)}(R))$ is a CN ring.\item[(2)] $C(L_{(s,t)}(R))$ is a ring consisting of elements having CN decompositions.\item[(3)] Assume that $R$ is a CN ring. If for any $\{a, d, f\}\subseteq R$ having a CN decomposition $a = x + p$, $d = y + q$ and $f = z + r$ with $\{x, y, z\}\subseteq C(R)$ and $\{p, q, r\}\subseteq$ nil$(R)$ satisfy $sx = sy$ and $ty = tz$, then $L_{(s, t)}(R)$ is a CN ring.
\end{enumerate}
\end{prop}
\begin{proof} (1) Assume that $R$ is a CN ring. Let $A = \begin{bmatrix}a&0&0\\0&a&te\\0&0&a\end{bmatrix}\in V_2(L_{(s,
t)}(R))$. There exist $c\in C(R)$ and $n\in$ nil$(R)$ such that $a = c + n$. Then $C = \begin{bmatrix}c&0&0\\0&c&0\\0&0&c\end{bmatrix}\in C(L_{(s, t)}(R))$ and $N = \begin{bmatrix}n&0&0\\0&n&te\\0&0&n\end{bmatrix}\in$ nil$(V_2(L_{(s, t)}(R)))$ and $A = C + N$ is the CN decomposition of $A$ in $V_2(L_{(s, t)}(R))$. For the inverse implication, let $r\in R$ and consider $A = \begin{bmatrix}r&0&0\\0&r&0\\0&0&r\end{bmatrix}\in V_2(L_{(s, t)}(R))$. There exist $C = \begin{bmatrix}a&0&0\\0&a&te\\0&0&a\end{bmatrix}\in C(V_2(L_{(s, t)}(R)))$ and $N = \begin{bmatrix}p&0&0\\0&r&tu\\0&0&v\end{bmatrix}\in$ nil$(V_2(L_{(s, t)}(R)))$. Then $a\in C(R)$ and $p\in$ nil$(R)$ and $r = a + n$ is the CN decomposition of $r$. Hence $R$ is a CN ring.\\
(2) Let $A = \begin{bmatrix}a&0&0\\sc&d&te\\0&0&f\end{bmatrix}\in CN_{(s,
t)}(R)$. Set $C = \begin{bmatrix}a&0&0\\0&d&0\\0&0&f\end{bmatrix}$ and $N = \begin{bmatrix}0&0&0\\sc&0&te\\0&0&0\end{bmatrix}$. By Lemma \ref{com}, $C\in C(L_{(s, t)}(R))$ and $N\in$ nil$(L_{(s, t)}(R)$. $A = C + N$ is the CN decomposition of $A$.\\
(3) Let $A = \begin{bmatrix}a&0&0\\sc&d&te\\0&0&f\end{bmatrix}\in L_{(s, t)}(R)$. Let $a = x + p$, $d = y + q$ and $f = z + r$ denote the CN decompositions of $a$, $d$ and $f$. By hypothesis $sx = sy$ and $ty = tz$. By (2) $A$ has a CN decomposition in $L_{(s, t)}(R)$ as $A = C + N$ where $C = \begin{bmatrix}x&0&0\\0&y&0\\0&0&z\end{bmatrix}\in C(L_{(s, t)}(R))$ and $N = \begin{bmatrix}p&0&0\\sc&q&te\\0&0&r\end{bmatrix}\in$ nil$(L_{(s, t)}(R))$.
\end{proof}
\begin{cor} Let $R$ be a ring. If $L_{(s, t)}(R)$ is a CN ring, then $R$ is a CN ring.
\end{cor}
\begin{proof} Assume that $L_{(s, t)}(R)$ is a CN ring and let $a\in R$ and $A = \begin{bmatrix}a&0&0\\0&a&0\\0&0&a\end{bmatrix}\in  L_{(s, t)}(R)$. By hypothesis there exist $C = \begin{bmatrix}x&0&0\\sy&z&tu\\0&0&v\end{bmatrix}\in C(L_{(s, t)}(R))$ and $N = \begin{bmatrix}n&0&0\\sc&m&te\\0&0&k\end{bmatrix}\in$ nil$(L_{(s, t)}(R))$ such that $A = C + N$ where $x\in C(R)$ and $n\in$ nil$(R)$. Then $a = x + n$ is the CN decomposition of $a$.
\end{proof}

There are CN rings such that $L_{(s, t)}(R)$ need not be a CN
ring.

\begin{ex}{\rm Let $R = \Bbb Z$ and
$A = \begin{bmatrix}1&0&0\\3&2&2\\0&0&3\end{bmatrix}\in  L_{(1,
1)}(R)$. Assume that $A = C + N$ is a CN decomposition of $A$. Since $A$ is neither central nor nilpotent, by
Lemma \ref{com}, we should get $A$ had a CN decomposition as $A = C + N$ where $C =
\begin{bmatrix}1&0&0\\0&1&0\\0&0&1\end{bmatrix}\in  C(L_{(1,
1)}(R))$ and $N =
\begin{bmatrix}x&0&0\\c&y&e\\0&0&z\end{bmatrix}\in$ nil$ (L_{(1,
1)}(R))$ where $\{x,
 y, z\}\subseteq$ nil$(\Bbb Z)$.  This leads us a contradiction in $\Bbb Z$.}
\end{ex}
\begin{prop} $R$ is CN ring if and only if so is $L_{(0, 0)}(R)$.
\end{prop}
\begin{proof} Note that $L_{(0, 0)}(R)$ is isomorphic to the ring $R\times R\times R$. By Proposition \ref{direct}, $\prod_{i\in I} R_i$ is a CN ring if and only if each $R_i$ is a CN ring for each $i\in I$.
\end{proof}
{\bf The rings $H_{(s,t)}(R)$ :} Let $R$ be a ring and  $s, t$ be in the
center of $R$. Let\begin{center} $H_{(s,t)}(R) = \left
\{\begin{bmatrix}a&0&0\\c&d&e\\0&0&f
\end{bmatrix}\in M_3(R)\mid a, c, d, e, f\in R, a - d = sc, d - f = te\right \}$.\end{center}
Then $H_{(s,t)}(R)$ is a subring of $M_3(R)$. Note that any
element $A$ of $H_{(s,t)}(R)$ has the form
$\begin{bmatrix}sc+te+f&0&0\\c&te+f&e\\0&0&f\end{bmatrix}$.
\begin{lem}\label{nil} Let $R$ be a ring, and let $s, t$ be in the center of $R$. Then the set of all nilpotent elements of $H_{(s, t)}(R)$ is \begin{center} nil$(H_{(s, t)}(R)) = \left
\{\begin{bmatrix}a&0&0\\c&d&e\\0&0&f\end{bmatrix}\in H_{(s,
t)}(R)\mid a, d, f\in  nil(R), c, e\in R\right \}$.\end{center}
\end{lem}
\begin{proof} Let $A = \begin{bmatrix}a&0&0\\c&d&e\\0&0&f\end{bmatrix}\in$ nil$(H_{(s, t)}(R))$. There exists a positive
integer $n$ such that $A^n = 0$. Then $a^n = d^n = f^n = 0$.
Conversely assume that $a^n = 0$, $d^m = 0$ and $f^k= 0$ for some
positive integers $n,m,k$. Let $p = max\{n, m, k\}$. Then $A^{2p} = 0$.
\end{proof}
\begin{lem}\label{mer} Let $R$ be a ring, and let $s$ and $t$ be central invertible in $R$. Then\\
$C(H_{(s, t)}(R)) = \left\{\begin{bmatrix}a& 0 & 0 \\c & d & e \\0 & 0 & f \end{bmatrix} \in H_{(s, t)}(R) \mid c,e,f\in C(R)\right \}$.
\end{lem}
\begin{proof} \cite[Lemma 3.1]{KKUH}. 
\end{proof}
\begin{thm}\label{guze} Let $R$ be a ring. $R$ is a CN ring if and only if $H_{(s, t)}(R)$ is a CN ring.
\end{thm}
\begin{proof} Assume that $R$ is a CN ring. Let $A = \begin{bmatrix}a&0&0\\c&d&e\\0&0&f\end{bmatrix}\in (H_{(s, t)}(R))$. Then $a = c_1 + n_1$, $d = c_2 + n_2$, $f = c_3 + n_3$, $c = c_4 + n_4$, $e = c_5 + n_5$ with $\{c_1, c_2, c_3, c_4, c_5\}\subseteq C(R)$, $\{n_1, n_2, n_3, n_4, n_5\}\subseteq $nil$ (R)$. Let $c_1 - c_2 = sc_4$, $c_2 - c_3 = tc_5$, $n_1 - n_2 = sn_4$ and $n_2 - n_3 = tn_5$ and $C = \begin{bmatrix}c_1&0&0\\c_4&c_2&c_5\\0&0&c_3\end{bmatrix}$ and $N = \begin{bmatrix}n_1&0&0\\n_4&n_2&n_5\\0&0&n_3\end{bmatrix}$. By Lemma \ref{mer}, $C\in C(H_{(s, t)}(R))$ and by Lemma \ref{nil}, $N\in$ nil$(H_{(s, t)}(R))$. Then $A = C + N$ is the CN decomposition of $A$.

Conversely, suppose that $H_{(s, t)}(R)$ is a CN ring. Let $a\in
R$. Then $A = \begin{bmatrix}a&0&0\\0&a&0\\0&0&a\end{bmatrix}\in
(H_{(s, t)}(R))$ and it has a CN decomposition $A = C + N$ where
$C = \begin{bmatrix}x& 0 & 0 \\y & z & u \\0 & 0 & v \end{bmatrix}
\in C(H_{(s, t)}(R))$ with $\{y, u, v\}\subseteq C(R)$ and $N =
\begin{bmatrix}n_1&0&0\\n_2&n_3&n_4\\0&0&n_5\end{bmatrix}\in$
nil$(H_{(s, t)}(R))$ with $\{n_1, n_3, n_5\}\subseteq$  nil$(R)$.
Then $a = x + n_1$ is a CN decomposition of $a$.
\end{proof}

\begin{prop} Uniquely nil clean rings, uniquely strongly nil clean rings, strongly nil *-clean rings are CN.
\end{prop}
\begin{proof} These classes of rings are abelian. Assume that $R$ is uniquely nil clean ring. Let $e$ be an idempotent in $R$. For any $r\in R$, $e + (re - ere)$ can be written in two ways as a sum of an idempotent and a nilpotent as
$e + (re - ere) = (e + (re -ere)) + 0 = e + (re - ere)$. Then $e = e + (re - ere)$ and $er - ere = 0$. Similarly, $e + (er - ere) = (e + (re -ere)) + 0 = e + (er - ere)$. Then $0 = re - ere = re - ere$. Hence $e$ is central.
\end{proof}

The converse of this result is not true.

\begin{ex}\label{ters}  The ring $H_{(0, 0)}(\Bbb Z)$ is $CN$ but not uniquely nil clean.
\end{ex}
\begin{proof} By Theorem \ref{guze}, $H_{(0, 0)}(\Bbb Z)$ is $CN$. Note that for $n\in \Bbb Z$ has a uniquely nil clean decomposition if and only if $n = 0$ or $n = 1$. Let $ A = \begin{bmatrix}a&0&0\\c&a&e\\0&0&a\end{bmatrix}\in H_{(0, 0)}(R)$ with $a\notin \{0, 1\}$. Assume that $A$ has a
 uniquely nil clean decomposition. There exist unique $E^2 = E =
 \begin{bmatrix}x&0&0\\y&x&u\\0&0&x\end{bmatrix}\in H_{(0, 0)}(R)$
and $N = \begin{bmatrix}g&0&0\\h&g&l\\0&0&g\end{bmatrix}\in$
$N(H_{(0, 0)}(R)$ such that $A = E + N$. Then $A$ has a uniquely
nil clean decomposition. So $a = x + g$ has a CN decomposition.
This is not the case for $a\in \Bbb Z$. Hence $H_{(0, 0)}(\Bbb Z)$
is not uniquely nil clean.
\end{proof}

{\bf Generalized matrix rings:} Let $R$ be ring and $s$ a central
element of $R$. Then  $\begin{bmatrix} R&R\\R&R\end{bmatrix}$
becomes a ring denoted by $K_s(R)$ with addition defined
componentwise and with multiplication defined in \cite{Kr} by
$$\begin{bmatrix} a_1&x_1\\y_1&b_1\end{bmatrix}\begin{bmatrix}
a_2&x_2\\y_2&b_2\end{bmatrix} = \begin{bmatrix}a_1
a_2+sx_1y_2&a_1x_2+x_1b_2\\y_1a_2+b_1y_2&sy_1x_2+b_1b_2\end{bmatrix}.$$
In \cite{Kr}, $K_s(R)$ is called a {\it generalized matrix ring
over $R$}.
\begin{lem}\label{inv} Let $R$ be a commutative ring. Then the following
hold.
\begin{enumerate}
\item[(1)] nil$(K_0(R)) = \left\{\begin{bmatrix}a&b\\c&d\end{bmatrix}\in K_0(R)\mid \{a, d\}\subseteq nil(R)\right\}$.\item[(2)] $C(K_0(R))$ consists of all scalar matrices.
\end{enumerate}
\end{lem}
\begin{proof} (1) Let $A  =  \begin{bmatrix}a&b\\c&d\end{bmatrix}\in$ nil$(K_0(R))$. Then $A^2 = \begin{bmatrix}a^2&b(a+d)\\c(a+d)&d^2\end{bmatrix}$, ..., \\
 ... , $A^{2^{n}} = \begin{bmatrix}a^{2^{n}}&\Sigma^n_{i=1} b(a^{2^{i-1}}+d^{2^{i-1}})\\\Sigma^n_{i=1} c(a^{2^{i-1}}+d^{2^{i-1}})&d^{2^n}\end{bmatrix}$. Hence $A\in$ nil$(K_0(R))$ if and only if $\{a, d\}\subseteq$ nil$(R)$.
\end{proof}

\begin{lem} Let $R$ be ring. Then $R$ is a CN ring if and only if $D_n(K_0(R))$ is a CN ring.
\end{lem}
\begin{proof} Necessity: We assume that $n = 2$. Let $A =
\begin{bmatrix}a&b\\0&a\end{bmatrix}\in D_2((K_0(R)))$. By
assumption $a = c_1 + n_1$ where $c_1\in C(R)$ and $n_1\in$
nil$(R)$. Let $C = \begin{bmatrix}c_1&0\\0&c_1\end{bmatrix}\in
C(D_2(K_0(R)))$ and $N =
\begin{bmatrix}n_1&b\\0&n_1\end{bmatrix}\in$ nil$D_2((K_0(R)))$. $A = C + N$ is the CN decomposition of $A$.\\Sufficiency:
Let $a\in R$. Then $A = \begin{bmatrix}a&0\\0&a\end{bmatrix}\in
D_2((K_0(R)))$ has a CN decomposition $A = C + N$ with $C =
\begin{bmatrix}c_1&0\\0&c_1\end{bmatrix}\in C(D_2((K_0(R))))$ and
$N = \begin{bmatrix}n_1&b_1\\0&n_1\end{bmatrix}\in$
nil$(D_2((K_0(R))))$ where $c_1\in C(R)$ and $n_1\in$ nil$(R)$. By
comparing components of matrices we get $a = c_1 + n_1$. It is a CN
decomposition of $a$.
\end{proof}
Note that $K_0(R)$ need not be a CN ring.
\begin{ex}{\rm  Let $A = \begin{bmatrix}1&0\\0&0\end{bmatrix}\in K_0(\Bbb Z)$ have a CN decomposition as $A = C + N$ where $C\in C(K_0(\Bbb Z))$ and $N \in$nil$ K_0(\Bbb Z))$. Then we should have $C = \begin{bmatrix}x&0\\0&x\end{bmatrix}$ and $N = \begin{bmatrix}1-x&0\\0&-x\end{bmatrix}$. These imply $x = 1$ or $x$ is nilpotent. A contradiction.}
\end{ex}


\begin{thebibliography}{000}



\bibitem{CU} H. Chen, S. Halicioglu, A. Harmanci and Y. Kurtulmaz,{\it On the Decomposition of a Matrix into the Sum of a Central Matrix and a Unit Matrix }, submitted for publication.

\bibitem{Di} A. J. Diesl, {\it Nil clean rings}, J. Algebra, 383(2013), 197-211.

\bibitem{Dr} J. L. Dorroh, {\it Concerning Adjunctions to Algebras}, Bull. Amer. Math. Soc., 38(1932), 85-88.


\bibitem{KKUH} H. Kose, Y. Kurtulmaz, B. Ungor and A. Harmanci, {\it Rings Having Normality in terms of the Jacobson
Radical}, Arab. J. Math., https://doi.org/10.1007/s40065-018-0231-7.
\bibitem{HTY} Y. Hirano, H. Tominaga and A. Yaqub, {\it On rings in which every element is uniquely expressable as a sum of a nilpotent element and a certain potent element}, Math. J. Okayama Univ. 30(1988), 33-40.

\bibitem{Kr} P. A. Krylov, {\it Isomorpism of generalized matrix rings}, Algebra and Logic, 47(4)(2008), 258-262.

\bibitem{LZ} C. Li and Y. Zhou, {\it On strongly *-clean rings}, J. Algebra Appl. 10(2011), 1363-1370.

%
\end{thebibliography}
\end{document}